\documentclass[11pt]{article}

\usepackage{amsmath,amssymb,amsfonts}
\usepackage{amsthm}
\usepackage{graphicx}
\usepackage{textcomp}
\usepackage{braket}
\usepackage{xcolor}
\usepackage{a4wide}
\usepackage{authblk}

\usepackage[ruled,vlined, linesnumbered]{algorithm2e}
\newtheorem{thm}{Theorem}
\newtheorem{cor}{Corollary}

\newtheorem{prop}{Proposition}

\newtheorem{rem}{Remark}

\newcommand{\Rset}{\mathbb{R}}
\newcommand{\Xset}{\mathbb{X}}

\begin{document}

\title{Distributionally robust possibilistic optimization problems} 

\author[1]{Romain Guillaume}
\author[2]{Adam Kasperski\footnote{Corresponding author}}
\author[3]{Pawe{\l} Zieli\'nski}

  \affil[1]{Universit{\'e} de Toulouse-IRIT Toulouse, France, \texttt{romain.guillaume@irit.fr}}
\affil[2]{
Wroc{\l}aw  University of Science and Technology, Wroc{\l}aw, Poland\\
            \texttt{adam.kasperski@pwr.edu.pl}}
\affil[3]{
Wroc{\l}aw  University of Science and Technology, Wroc{\l}aw, Poland\\
            \texttt{pawel.zielinski@pwr.edu.pl}}

\date{}

\maketitle

 \begin{abstract}
 In this paper a class of optimization problems with uncertain linear constraints is discussed. It is assumed that the constraint coefficients are random vectors whose probability distributions are only partially known. Possibility theory is used to model the imprecise probabilities. 
In one of  the interpretations,
 a possibility distribution (a membership function of a fuzzy set) in the set of coefficient realizations  induces a necessity measure, which in turn defines a family of probability distributions  in this set. The distributionally robust approach is then used to transform the imprecise constraints into deterministic counterparts. Namely, the uncertain left-had side of each constraint is replaced with the expected value with respect to the worst probability distribution that can occur. It is shown how to represent the resulting problem by using linear or second order cone constraints. This leads to  problems which are computationally
  tractable for a wide class of optimization models, in particular for linear programming.
 \end{abstract}


\section{Introduction}

In this paper we consider the following uncertain optimization problem:
\begin{equation}
\label{pf}
	\begin{array}{llll}
		\min & \pmb{c}^T\pmb{x} \\
			& \tilde{\pmb{a}}_i^T\pmb{x}\leq b_i & i\in [m]\\
			& \pmb{x}\in \mathbb{X}.
	\end{array}
\end{equation}
In formulation~(\ref{pf}), $\pmb{c}\in \mathbb{R}^n$ is a vector of objective function coefficients,  $\pmb{b}\in \mathbb{R}^m$ is a vector of constraint right-hand sides, and $\tilde{\pmb{a}}_i$ is a vector of  uncertain coefficients of the $i$th constraint. In the following we will use the notation $[m]=\{1,\dots,m\}$.
We assume that  $\tilde{\pmb{a}}_i$ is a random vector in $\mathbb{R}^n$ whose realization is unknown when a solution to~(\ref{pf}) is determined. A particular realization $\pmb{a}_i\in \mathbb{R}^n$ of $\tilde{\pmb{a}}_i$ is called a \emph{scenario}. After replacing $\tilde{\pmb{a}}_i$ with scenario $\pmb{a}_i$, for each $i\in [m]$, we get a deterministic counterpart of~(\ref{pf}). The true probability distribution for $\tilde{\pmb{a}}_i$ is only partially known. In particular, the set of possible scenarios $\mathcal{U}_i\subseteq \mathbb{R}^n$ is provided, which is the support of $\tilde{\pmb{a}}_i$ or its reasonable approximation.
Finally, $\mathbb{X}$ is a nonempty and bounded subset of $\mathbb{R}^n$ which restricts the domain of decision variables $\pmb{x}$. If  $\mathbb{X}$ is a polyhedron, for example, $\mathbb{X}=\{\pmb{x}\in \mathbb{R}^n: \pmb{L}\leq \pmb{x}\leq \pmb{U}\}$, for some $\pmb{L}, \pmb{U}\in \mathbb{R}^n$, then we get  uncertain linear programming problem. If, additionally, $\mathbb{X}\subseteq \{0,1\}^n$, then we get uncertain combinatorial optimization problem.

The method of handling the uncertain constraints and solving~(\ref{pf}) depends on the information available about $\tilde{\pmb{a}}_i$. If $\tilde{\pmb{a}}_i$ is a random vector with a known probability distribution, then the $i$th constraint can be replaced with a stochastic \emph{chance constraint} of the form 
$${\rm P}(\tilde{\pmb{a}}_i^T\pmb{x}\leq b_i)\geq 1-\epsilon_i,$$
 where $\epsilon_i\in [0,1)$ is a given risk level~\cite{CC59, KM05}.
If the uncertainty set $\mathcal{U}_i$ is the only information provided with $\tilde{\pmb{a}}_i$, then the imprecise constraint $\tilde{\pmb{a}}_i^T\pmb{x}\leq b_i$ can be replaced with the following
\emph{strict robust constraint}~\cite{BN09}:
\begin{equation}
\label{rc}
\max_{\pmb{a}_i\in \mathcal{U}_i} \pmb{a}_i^T\pmb{x}\leq b_i.
\end{equation}

 Choosing appropriate uncertainty set $\mathcal{U}_i$ is crucial in~(\ref{rc}). For the \emph{discrete uncertainty}, $\mathcal{U}_i$ consists of $K$ explicitly listed scenarios~\cite{KY97}. These scenarios correspond to some events which influence the value of $\tilde{\pmb{a}}_i$ or can be the result of sampling the random vector $\tilde{\pmb{a}}_i$. For the \emph{interval uncertainty}~\cite{KY97}, one provides an interval for each component $\tilde{a}_{ij}$ of $\tilde{\pmb{a}}_i=(\tilde{a}_{i1},\dots,\tilde{a}_{in})$ and $\mathcal{U}_i$ is the Cartesian product  of these intervals (a hyperrectangle in $\mathbb{R}^n$). In order to cut off the extreme values of this hyperrectangle, whose probability of occurrence can be negligible, one can add a \emph{budget} to $\mathcal{U}_i$. This budget typically limits deviations of  scenarios from some nominal scenario $\hat{\pmb{a}}_i\in \mathcal{U}_i$~\cite{BS04, NO13}. In another approach $\mathcal{U}_i$ is specified as an ellipsoid centered at some nominal scenario $\hat{\pmb{a}}_i$~\cite{BN09}.  Using the ellipsoidal uncertainty one can model correlations between the components of $\tilde{\pmb{a}}_i$.  One can also construct $\mathcal{U}_i$ from available data~\cite{BGK17}, which leads to various \emph{data-deriven robust models}. 
 
Another approach to handle the uncertainty consists in modeling the uncertain coefficients in~(\ref{pf}) by using \emph{fuzzy sets}. There are a lot of concepts of solving  fuzzy optimization problems 
(see, e.g.,~\cite{IRTV03,RV02,LH92,R06, ST90, QSC12, PZT11}) -- for 
 a comprehensive description of them 
 we refer the reader to~\cite{LK10}. In one of the most 
common class of models, $\tilde{\pmb{a}}_i$ is a vector of fuzzy intervals whose membership functions are interpreted as \emph{possibility distributions}
 for the components of~$\tilde{\pmb{a}}_i$. A description of \emph{possibility theory} that 
 offers a general framework of dealing with imprecise or incomplete knowledge can be found, for example, in~\cite{DP88}.
 It uses two dual \emph{possibility} and \emph{necessity measures} to handle the uncertainty. In the possibilistic setting one can replace the uncertain constraints in~(\ref{pf}) with \emph{fuzzy chance constraints} of the form~\cite{IR00, LIU01,KZ21}:
$$
\label{necc}
{\rm N}(\tilde{\pmb{a}}^T_i\pmb{x}\leq b_i)\geq 1-\epsilon_i,
$$
where ${\rm N}$ is a \emph{necessity measure} induced by a possibility distribution for $\tilde{\pmb{a}}_i$.

If a partial information about the probability distribution of $\tilde{\pmb{a}}_i$ is available, then the so-called \emph{distributionally robust approach} can be used (see, e.g.,~\cite{GS10, WKS14, DY10}). 
In this approach it is assumed
 that  the true probability distribution ${\rm P}$ for $\tilde{\pmb{a}}_i$ belongs to the so-called \emph{ambiguity set}  $\mathcal{P}_i$ of probability distributions. For example, one can consider all probability distributions  given some bounds on its mean and covariance matrix~\cite{DY10}. Alternatively, the ambiguity set can be specified by providing a family of confidence sets (this approach will be adopted in this paper)~\cite{WKS14}. Using the distributionally robust approach,  one can consider the following counterpart of the imprecise constraint $\tilde{\pmb{a}}_i^T\pmb{x}\leq b_i$:
\begin{equation}
\label{drc}
\max_{{\rm P}\in \mathcal{P}_i}  \mathrm{\mathbf{E}}_{\rm P}[\tilde{\pmb{a}}_i^T\pmb{x}]\leq b_i,
\end{equation}
where the left hand side is the expected value of the random variable $\tilde{\pmb{a}}_i^T\pmb{x}$ under a worst probability distribution  ${\rm P}\in\mathcal{P}_i$. Observe that~(\ref{drc}) reduces to~(\ref{rc}) if $\mathcal{P}_i$ contains all probability distributions with the support $\mathcal{U}_i$. 

 In this paper we propose a new approach in the area of fuzzy optimization. Will show how the distributionally robust optimization can be naturally used in the setting of possibility theory. A preliminary version of the paper has appeared in~\cite{GKZ21a}.
We will start by defining a possibility distribution (a membership function of a fuzzy set) for $\tilde{\pmb{a}}_i$. This can be done by using both available data or experts knowledge.
Following the interpretation of possibility distribution~\cite{BD06, DD06}, 
we will define an ambiguity set of probability distributions $\mathcal{P}_i$ for $\tilde{\pmb{a}}_i$.  Finally, we will apply constraints of the type~(\ref{drc}) to compute a solution. We will show that the model of uncertainty assumed leads to a family of linear or second order cone constraints. So, the resulting problem is tractable for important cases of~(\ref{pf}), for example for the class of linear programming problems. \begin{rem}
\label{rcbprec}
If one uses~(\ref{drc}), then the assumption that $\pmb{c}$ and $b_i$, $i\in [m]$, are precise in~(\ref{pf}) causes no loss of generality. Indeed, if $\tilde{\pmb{c}}$ is imprecise, then we can minimize an auxiliary variable $t$, subject to the imprecise constraint $\tilde{\pmb{c}}^T\pmb{x}-t\leq 0$
($t$   reflects the possible realizations of objective function values).
If $\tilde{b}_i$ is imprecise, then the $i$th constraint can be rewritten as $ \tilde{\pmb{a}}_i^T\pmb{x}- \tilde{b}_i x_{n+1}\leq 0$, where $x_{n+1}$ is an auxiliary variable such that $x_{n+1}=1$. In both cases, we can replace the imprecise constraint with a one of the form~(\ref{drc}).
\end{rem}

This paper is organized as follows. In Section~\ref{secposs} we recall basic notions of possibility theory, in particular its probabilistic interpretation assumed in this paper. In Section~\ref{secdiscr} we discuss the discrete uncertainty representation in which a possibility degree for each scenario is specified. Then the ambiguity set $\mathcal{P}_i$ contains a family of discrete probability distributions with a finite support $\mathcal{U}_i$.
We will show that~(\ref{drc}) is then equivalent to a family of linear constraints. In Section~\ref{secint} we consider the interval uncertainty representation, in which unknown probability distribution for $\tilde{\pmb{a}}_i$ has a compact continuous support. We propose a method of defining a continuous possibility distribution for $\tilde{\pmb{a}}_i$, which can be built by using available data or experts knowledge. We show that~(\ref{drc}) is then equivalent to a family of second order cone constraints.

\section{Possibility theory and imprecise probabilities}
\label{secposs}

Let $\Omega$ be a set of alternatives. A primitive object of possibility theory (see, e.g.,~\cite{BD06}) is a \emph{possibility distribution} $\pi: \Omega\rightarrow [0,1]$, which assigns to each element $u\in \Omega$ a \emph{possibility degree} $\pi(u)$. We only assume that $\pi$ is \emph{normal}, i.e. there is $u\in \Omega$ such that $\pi(u)=1$. Possibility distribution can be built by using available data or experts opinions (see, e.g.,~\cite{DP88}).
A possibility distribution $\pi$ induces the following possibility and necessity measures in~$\Omega$:
$$\Pi(A)=\sup_{u\in A} \pi(u),\; A \subseteq \Omega.$$
$${\rm N}(A)=1-\Pi(A^c)=1-\sup_{u\in A^c} \pi(u),\; A\subseteq \Omega,$$
where $A^c=\Omega\setminus A$ is the complement of $A$.
In this paper we assume that the possibility distribution $\pi$ represents uncertainty in $\Omega$, i.e. some knowledge about uncertain quantity $\tilde{u}$ taking values in $\Omega$. 
We now recall, following~\cite{DD06}, a probabilistic interpretation of the pair $[\Pi, {\rm N}]$ induced by the possibility distribution $\pi$.
 Define 
 \begin{align}
 \mathcal{P}(\pi)&=\{{\rm P}: \forall A \text{ measurable } {\rm N}(A)\leq {\rm P}(A)\}\nonumber\\
                         &=\{{\rm P}: \forall A \text{ measurable } \Pi(A)\geq {\rm P}(A)\}.\label{defpm}
 \end{align} 
 In this case $\sup_{{\rm P}\in \mathcal{P}(\pi)} {\rm P}(A)=\Pi(A)$ and $\inf_{{\rm P}\in \mathcal{P}(\pi)} {\rm P}(A)={\rm N}(A)$ (see~\cite{DD06, DP92, CA99}).  So, the possibility distribution $\pi$ in $\Omega$ encodes a family of probability measures in $\Omega$. Any probability measure ${\rm P}\in \mathcal{P}(\pi)$ is said to be \emph{consistent with}~$\pi$ and for any event $A\subseteq \Omega$ the inequalities
\begin{equation} 
\label{eqnpp}
 {\rm N}(A)\leq {\rm P}(A)\leq \Pi(A)
 \end{equation}
 hold.
  A detailed discussion on the expressive power of~(\ref{eqnpp}) can also be found in~\cite{TMD13}.

 Observe that $\mathcal{P}(\pi)$ is nonempty due the assumption that $\pi$ is normalized. Indeed, the probability distribution such that ${\rm P}(\{u\})=1$ for some $u\in \Omega$ such that $\pi(u)=1$ is in $\mathcal{P}(\pi)$. To see this consider two cases. If $u\notin A$, then ${\rm N}(A)=0$ and ${\rm P}(A)\geq {\rm N}(A)=0$. If $u\in A$, then ${\rm P}(A)=1\geq {\rm N}(A)$. 
 
\section{Discrete  uncertainty model}
\label{secdiscr}

Consider uncertain constraint $\tilde{\pmb{a}}^T\pmb{x}\leq b$, where   $\tilde{\pmb{a}}$ is a random vector in $\mathbb{R}^n$ with an unknown probability distribution (in order to simplify notation we skip the index~$i$ in the constraint). 
In this section we assume that the support $\mathcal{U}$ of $\tilde{\pmb{a}}$ is finite, i.e. $\mathcal{U}=\{\pmb{a}_1,\dots,\pmb{a}_K\}$ is the set of all possible, explicitly listed scenarios (realizations of $\tilde{\pmb{a}}$) which can occur with a positive probability. We thus consider the discrete uncertainty model~\cite{KY97,KZ16b}.
Let $\mu$ be  a membership function of a  fuzzy set  in~$\mathcal{U}$, $\mu:\mathcal{U}\rightarrow [0,1]$.
We will interpret~$\mu$ as a possibility distribution~$\pi$  in the scenario set $\mathcal{U}$,
i.e. $\pi=\mu$. In order to simplify presentation, we will identify each scenario $\pmb{c}_i\in \mathcal{U}$ with its index~$i\in [K]$.
We can thus assume 
 that $\pi$ is a possibility distribution
  in the index set~$[K]$, $\pi: [K]\rightarrow [0,1]$.
   The value of $\pi(i)$ is the \emph{possibility degree} for scenario $\pmb{a}_i$.
 Assume also w.l.o.g. that $\pi(1)\geq \pi(2)\geq \dots \geq \pi(K)$ and thus $\pi(1)=1$. 
 
Given $\pi$, we can now
 construct an ambiguity set~$\mathcal{P}(\pi)$ of probability distributions for~$\tilde{\pmb{a}}$ by using~(\ref{defpm}).
  Because $\mathcal{P}(\pi)$ contains only discrete probability distributions in $[K]$, we get

 \[
 \mathcal{P}(\pi)\subseteq \{\pmb{p}\in [0,1]^K \,:  \sum_{i\in [K]} p_i=1\}
 \]
and we can describe~$\mathcal{P}(\pi)$ by the following system of linear constraints:
\begin{equation}
\label{e2}
	\begin{array}{llll}
		\displaystyle \sum_{i\in A} p_i\geq 1-\max_{i\notin A} \pi(i) &\forall\, A\subset [K], |A|\geq 1 \\
		\displaystyle \sum_{i\in [K]} p_i=1\\
		p_i\geq 0 & i\in [K]
	\end{array}
\end{equation}
In description~(\ref{e2})
the first set of constraints model the condition ${\rm P}(A)\geq 1-\Pi(A^c)= {\rm N}(A) \text{ for all events } A\subset [K]$,
($\Omega=[K]$).
Notice that the formulation~(\ref{e2}) has
 exponential number of constraints. 
In the following, we will show how to reduce the number  of constraints to $O(K)$.

Let us partition the set of indices $[K]$ into $\ell$ disjoint sets $I_1,\dots, I_{\ell}$ such that the elements of $I_j$ have the same possibility degrees and the possibility degrees of the elements in $I_j$ are greater than the possibility degrees of the elements in $I_k$ for any $j>k$. For example, let $\pmb{\pi}=(1,1,0.5,0.5,0.3, 0.3, 0.3, 0.1)$ be a possibility distribution 
(a fuzzy set)
in~$[K]$ for $K=8$
(i.e. a possibility distribution in the set of eight scenarios). Then $I_1=\{1,2\}$, $I_2=\{3,4\}$, $I_3=\{5,6,7\}$, $I_4=\{8\}$. Let $\pi^j$ be the possibility degree of the elements in $I_j$. In the example we have $\pi^1=1$, $\pi^2=0.5$, $\pi^3=0.3$ and $\pi^4=0.1$.
\begin{prop}
\label{prop1}
	System~(\ref{e2}) is equivalent to the following system of constraints:
	\begin{equation}
\label{e3}
	\begin{array}{llll}
		\displaystyle \sum_{i\in I_1\cup\dots \cup I_j} p_i\geq 1-\pi^{j+1} & j\in[\ell-1]\\
		\displaystyle \sum_{i\in [K]} p_i=1\\
		p_i\geq 0 & i\in [K]
	\end{array}
\end{equation}
\end{prop}
\begin{proof}
It is easy to verify that each feasible solution to~(\ref{e2}) is also feasible to~(\ref{e3}). Indeed,  for $A= I_1\cup\dots \cup I_j$, $j\in [\ell-1]$, we get $\max_{i\notin A} \pi(i)=\pi^{j+1}$. So, (\ref{e3}) is a subset of the constraints~(\ref{e2}).
 Assume that $\pmb{p}$ is feasible to~(\ref{e3}). Consider any set $A\subseteq [K]$, $|A|\geq 1$. If $A=[K]$, then $\pmb{p}$ is feasible to~(\ref{e2}). Assume that $A\subset [K]$ and let $k\in[\ell-1]$ be the first index such that $I_k\not\subseteq A$.  If $k=1$, then ${\rm N}(A)=0$ and the constraint in~(\ref{e2}) holds for $A$. Assume that $k>1$. The inequality from~(\ref{e3})
$$ \sum_{i\in I_1\cup\dots \cup I_{k-1}} p_i\geq 1-\pi^{k}$$
implies
$$\sum_{i\in A} p_i \geq \sum_{i\in I_1\cup\dots \cup I_{k-1}} p_i\geq 1-\pi^k=1-\max_{i\notin A} \pi(i),$$
where the last equality results from the fact that $I_k\not \subseteq A$.
In consequence the constraint in~(\ref{e2}) holds for $A$.
\end{proof}

For the sample possibility distribution $\pmb{\pi}=(1,1,0.5,0.5,0.3, 0.3, 0.3, 0.1)$, the ambiguity set $\mathcal{P}(\pi)$ can be described by the following system of constraints:
$$
\begin{array}{llll}
		p_1+p_2\geq 0.5 \\
		p_1+p_2+p_3+p_4\geq 0.7\\
		p_1+p_2+p_3+p_4+p_5+p_6+p_7\geq 0.9\\
		p_1+p_2+p_3+p_4+p_5+p_6+p_7+p_8=1\\
		p_i\geq 0 & i\in [K]
	\end{array}
$$
Using Proposition~\ref{prop1}, the value of $\max_{{\rm P}\in \mathcal{P}(\pi)} \mathrm{\mathbf{E}}_{\rm P}[\tilde{\pmb{a}}^T\pmb{x}]$, for a given~$\pmb{x}$, can be computed by using the following linear program:
\begin{equation}
\label{mod01}
	\begin{array}{lllll}
			\max \displaystyle \sum_{i\in [K]} p_i \pmb{a}^T_i\pmb{x}\\
				\displaystyle \sum_{i\in I_1\cup\dots \cup I_j} p_i\geq 1-\pi^{j+1} & j\in [\ell-1] \\
		\displaystyle \sum_{i\in [K]} p_i=1\\
		p_i\geq 0 & i\in [K]\\
	\end{array}
\end{equation}
\begin{thm}
\label{thdc}
Given $\pmb{x}\in \mathbb{X}$, the constraint
	\begin{equation}
\label{drc00}
\max_{{\rm P}\in \mathcal{P}(\pi)} \mathrm{\mathbf{E}}_{\rm P}[\tilde{\pmb{a}}^T\pmb{x}]\leq b
\end{equation}
is equivalent to the following system of linear constraints:
\begin{equation}
\label{mod02}
	\begin{array}{lllll}
			 \displaystyle \beta+\sum_{j\in [\ell-1]} \alpha_j (\pi^{j+1}-1)\leq b\\
				\displaystyle \beta-\sum_{j\in [\ell-1]:  i\in I_j} \alpha_j \geq \pmb{a}_i^T \pmb{x} & i\in [K]\\
				\alpha_j\geq 0 & j\in [\ell-1]\\
				\beta\in \Rset &
	\end{array}
\end{equation}
\end{thm}
\begin{proof}
	For a fixed $\pmb{x}\in \mathbb{X}$,~(\ref{mod01}) is a linear programming problem, whose dual is:
	\begin{equation}
\label{mod03}
	\begin{array}{lllll}
			 \min & \displaystyle \beta+\sum_{j\in [\ell-1]} \alpha_j (\pi^{j+1}-1)\\
				& \displaystyle \beta-\sum_{j\in [\ell-1]:  i\in I_j} \alpha_j \geq \pmb{a}_i^T \pmb{x} & i\in [K]\\
				& \alpha_j\geq 0 & j\in [\ell-1]\\
				&\beta\in \Rset&
	\end{array}
\end{equation}
where $\beta$ and $\alpha_j$ are dual variables.
Now the strong duality for linear programming shows that
the optimal objective values of~(\ref{mod01}) and (\ref{mod03})
are the same (see, e.g.,~\cite[Theorem 3.1]{PS98}).
This gives us~(\ref{mod02}).
\end{proof}

Consider two boundary cases of the proposed model. If $\pi(k)=1$ for some $k\in [K]$, and $\pi(i)=0$ for each $i\neq k$, then $\pmb{a}_k$ is the only scenario which can occur. The constraint~(\ref{drc}) reduces then to 
$\pmb{a}_k^T\pmb{x}\leq b$. If $\pi(i)=1$ for all $i\in [K]$, then
$\mathcal{P}(\pi)=\{\pmb{p}\in [0,1]^K:  \sum_{i\in [K]} p_i=1\}$ and~(\ref{drc}) becomes the strict robust constraint~(\ref{rc}) for $\mathcal{U}=\{\pmb{a}_1,\dots,\pmb{a}_K\}$. Theorem~\ref{thdc} yields the following corollary.

\begin{cor}
If $\Xset$ is a polyhedron described by a system of linear constraints
((\ref{pf}) is an uncertain linear programming problem).
Then the  deterministic counterpart of~(\ref{pf}) with the constraints of
the type (\ref{drc00}) and in consequence of the type (\ref{mod02})
is a linear programming problem.
\label{cordc}
\end{cor}

Corollary~\ref{cordc} shows that
the resulting deterministic counterpart of uncertain linear programming problem~(\ref{pf}),
under the discrete  uncertainty model,
is polynomially solvable
and thus can be solved efficiently by  standard  off-the-shelf solvers.

\section{Interval uncertainty model}
\label{secint}
Let us 
 again  focus on an uncertain constraint $\tilde{\pmb{a}}^T\pmb{x}\leq b$, where   $\tilde{\pmb{a}}=(\tilde{a}_1,\dots,\tilde{a}_n)$
is a random vector in $\mathbb{R}^n$ with  an unknown  probability distribution having now a compact  continuous support $\mathcal{U}\subseteq \Rset^n$.  We will start by describing the support $\mathcal{U}$. Next, we will propose a continuous possibility distribution $\pi$ in $\mathcal{U}$ which will induce an ambiguity set of probability distributions $\mathcal{P}(\pi)$ according to~(\ref{defpm}). This possibility distribution can be provided by experts or can be estimated by using available data.
 Assume that for each component $\tilde{a}_j$ of $\tilde{\pmb{a}}$ a \emph{nominal} (expected, the most probable) value $\hat{a}_j$ and an interval $[\hat{a}_j-\underline{a}_j,\hat{a}_j+\overline{a}_j]$ containing possible values of $\tilde{a}_j$ are known. 
 Let
$$\mathcal{B}(0)=[\hat{a}_1-\underline{a}_1,\hat{a}_1+\overline{a}_1]\times\dots\times [\hat{a}_n-\underline{a}_n,\hat{a}_n+\overline{a}_n],$$
i.e. $\mathcal{B}(0)$ is a hyperrectangle containing possible scenarios $\pmb{a}$ (realizations of $\tilde{\pmb{a}}$).
The components of $\tilde{\pmb{a}}$ are often correlated. Also, the probability that a subset of components will take the extreme values can be very small or even~0.  So, $\mathcal{B}(0)$ is typically an overestimation of  the support $\mathcal{U}$ of $\tilde{\pmb{a}}$.
In robust optimization some  limit for the deviations of scenarios from the nominal vector $\hat{\pmb{a}}$ is often assumed, i.e. $\delta(\pmb{a})\leq \Gamma$, where $\delta(\pmb{a})$ is a deviation prescribed. The parameter $\Gamma\geq 0$ is called a \emph{budget} and it controls the amount of uncertainty in $\mathcal{U}$.  One can define, for example, $\delta(\pmb{a})=||\pmb{a}-\hat{\pmb{a}}||_{1}$, which leads to the continuous budgeted uncertainty~\cite{NO13}. In this paper, following~\cite{BN09, BN99} we use 
$$\delta(\pmb{a})= ||\pmb{B}(\pmb{a}-\hat{\pmb{a}})||_{2},$$
where $\pmb{B}$ is a given $n\times n$ matrix.
Notice that $\delta(\pmb{a})=(\pmb{a}-\hat{\pmb{a}})^T\pmb{\Sigma}(\pmb{a}-\hat{\pmb{a}})$, where $\pmb{B}=\pmb{\Sigma}^{\frac{1}{2}}$ for some postive semidefinite matrix $\pmb{\Sigma}$. For example, $\pmb{\Sigma}$ can be an estimation of the covariance matrix for $\tilde{\pmb{a}}$ (see, e.g.,~\cite{BS04r}).
Therefore
$$\mathcal{E}(0)=\{\pmb{a}\in \mathbb{R}^n: \delta(\pmb{a})\leq \Gamma\}$$
is an ellipsoid in $\mathbb{R}^n$. We postulate that $\mathcal{B}(0)\cap \mathcal{E}(0)$ contains all possible realizations of $\tilde{\pmb{a}}$, i.e. ${\rm P}(\mathcal{B}(0)\cap \mathcal{E}(0))=1$. Hence $\mathcal{B}(0)\cap \mathcal{E}(0)$ is the estimation of the support of $\tilde{\pmb{a}}$. In the following, we will construct a possibility distribution for $\tilde{\pmb{a}}$, which can be easily built  from  $\hat{a}_j, \underline{a}_j, \overline{a}_j$, $j\in [n]$, matrix $\pmb{B}$ and the budget~$\Gamma$.

 \begin{figure*}[ht!]
\centering
\includegraphics[width=12cm]{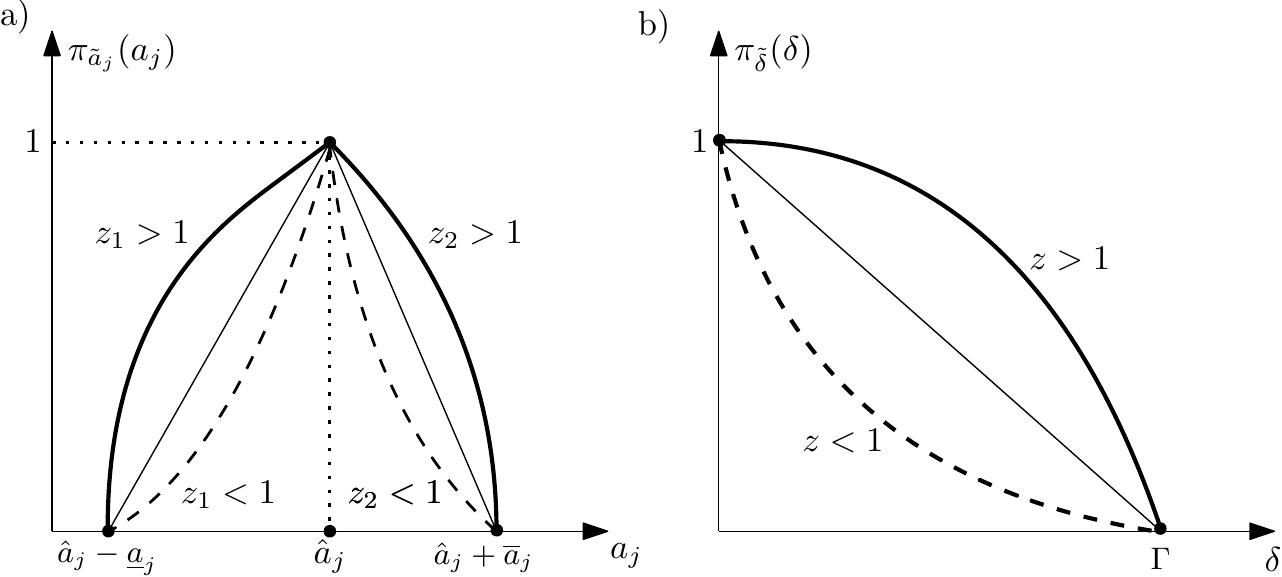}
\caption{Fuzzy intervals $\braket{\hat{a},\underline{a},\overline{a}}_{z_1-z_2}$ and $\braket{0,0,\Gamma}_{z}$ modeling the possibility distributions for $\tilde{a}_j$ and $\tilde{\delta}$. } \label{fig1}
\end{figure*}

Let $\pi_{\tilde{a}_j}$ be a possibility distribution for the component $\tilde{a}_j$ of $\tilde{\pmb{a}}$, such that that $\pi_{\tilde{a}_{j}}(\hat{a}_{j})=1$, $\pi_{\tilde{a}_j} (\hat{a}_j-\underline{a}_j)=\pi_{\tilde{a}_j}(\hat{a}_j+\overline{a}_j)=0$,  $\pi_{\tilde{a}_j}$ is continuous, strictly increasing in $[\hat{a}_j-\underline{a}_j,\hat{a}_j]$ and continuous strictly decreasing in $[\hat{a}_j,\hat{a}_j+\overline{a}_j]$. One can identify the possibility distribution $\pi_{\tilde{a}_j}$ with membership function~$\mu_{\tilde{a}_j}$ of  a fuzzy interval 
 $\braket{\hat{a}_j,\underline{a}_j,\overline{a}_j}_{z_1-z_2}$ shown in Figure~\ref{fig1}a (i.e. $\pi_{\tilde{a}_j}=\mu_{\tilde{a}_j}$).
 Let 
$$\tilde{a}_j(\lambda)=\{a_j: \pi_{\tilde{a}_j}(a_j)\geq \lambda\}\;\; \lambda\in (0,1]$$
be the $\lambda$-cut of~$\tilde{a}_j$.
It is easy to check that $\tilde{a}_j(\lambda)=[\underline{a}_j(\lambda),\overline{a}_j(\lambda)]$ 
is an interval
for each fixed  $\lambda\in (0,1]$.
Set $[\underline{a}_j(0),\overline{a}_j(0)]=[\hat{a}_i-\underline{a}_j, \hat{a}_j+\overline{a}_j]$. 
The function $\underline{a}_j(\lambda)$, the left profile of $\pi_{\tilde{a}_j}$, is continuous and strictly increasing and the function $\overline{a}_j(\lambda)$, the right profile of $\pi_{\tilde{a}_j}$,
 is continuous and strictly decreasing for $\lambda\in [0,1]$. 
For the fuzzy interval $\braket{\hat{a}_j,\underline{a}_j,\overline{a}_j}_{z_1-z_2}$, the  $\lambda$-cut of~$\tilde{a}_j$ can be computed by using the following formula:
$$\tilde{a}_j(\lambda)=[\hat{a}_j-\underline{a}_j(1-\lambda^{z_1}), \hat{a}_j+\overline{a}_j(1-\lambda^{z_2})]\; \lambda \in [0,1].$$
 
 We will assume that $\underline{a}_j, \overline{a}_j, z_1, z_2>0$. Then $\tilde{a}_j(\lambda)$ are nondegenerate intervals (they have nonempty interiors) for each $\lambda\in [0,1)$.
The numbers $z_1$ and $z_2$ control the amount of uncertainty for $\tilde{a}_j$ below and over $\hat{a}_j$, respectively. In particular, when $z_1,z_2\rightarrow 0$, then $\tilde{a}_j$ becomes $\hat{a}_j$. On the other hand, if $z_1,z_2\rightarrow \infty$, then $\tilde{a}_j$ becomes the closed interval $[\hat{a}_j-\underline{a}_j, \hat{a}_j+\overline{a}_j]$ (see Figure~\ref{fig1}).

The deviation $\delta(\pmb{a})$ depends on an unknown realization of $\tilde{\pmb{a}}$, so we can treat the deviation as an uncertain quantity $\tilde{\delta}$. We can now build the following possibility distribution for $\tilde{\delta}$:  $\pi_{\tilde{\delta}}(0)=1$, $\pi_{\tilde{\delta}}$ is continuous and strictly decreasing in $[0,\Gamma]$ and $\pi_{\tilde{\delta}}(\delta)=0$ if $\delta\notin [0,\Gamma]$. Let 
$$\tilde{\delta}(\lambda)=\{\delta\in \mathbb{R}^n: \pi_{\tilde{\delta}}(\delta)\geq \lambda\}=[0,\overline{\delta}(\lambda)] \;\; \lambda\in (0,1]$$
be the $\lambda$-cut of~$\tilde{\delta}$.
Fix $\overline{\delta}(0)=\Gamma$. We can identify $\pi_{\tilde{\delta}}$ with membership function
$\mu_{\tilde{\delta}}$
of a fuzzy interval~$\tilde{\delta}$ given, for example, in the form of $\braket{0,0,\Gamma}_{z}$ (see Figure~\ref{fig1}b).
The $\lambda$-cut of $\tilde{\delta}$ is then $\tilde{\delta}(\lambda)=[0,\Gamma(1-\lambda^z)]$ for $\lambda \in [0,1]$.

We now build the following joint possibility distribution $\pi:\mathbb{R}^n\rightarrow [0,1]$, for $\tilde{\pmb{a}}$:
\begin{equation}
\label{jpdi}
\pi(\pmb{a})=\min\{\pi_{\tilde{a}_1}(a_1),\dots, \pi_{\tilde{a}_n}(a_n), \pi_{\tilde{\delta}}(\delta(\pmb{a}))\}.
\end{equation}
The value of $\pi(\pmb{a})$ is the possibility degree for scenario $\pmb{a}\in \mathbb{R}^n$. 
 The first part of the possibility distribution, i.e. $\min\{\pi_{\tilde{a}_1}(a_1),\pi_{\tilde{a}_2}(a_2),\dots, \pi_{\tilde{a}_n}(a_n)\}$, is built from the possibility distributions of non-interacting components~\cite{DDC09}. The second part 
 $\pi_{\tilde{\delta}}(\delta(\pmb{a}))$ is added to model interactions between the components of~$\tilde{\pmb{a}}$.
Define the $\lambda$-cut of $\tilde{\pmb{a}}$
$$\mathcal{C}(\lambda)=\{\pmb{a}\in \mathbb{R}^n: \pi(\pmb{a})\geq \lambda\},\;\; \lambda\in (0,1].$$
It is easily seen that
\begin{align*}
\mathcal{C}(\lambda)=\{\pmb{a}\in\mathbb{R}^n \,:\, &a_j\in [\underline{a}_j(\lambda),\overline{a}_j(\lambda)], \; j\in [n],\\
& ||\pmb{B}(\pmb{a}-\hat{\pmb{a}})||_{2}\leq \overline{\delta}(\lambda)\}\; \lambda\in (0,1].
\end{align*}
Set $\mathcal{C}(0)=\mathcal{B}(0)\cap \mathcal{E}(0)$.
Observe that $\mathcal{C}(1)=\{\hat{\pmb{a}}\}$. Furthermore $\mathcal{C}(\lambda)$, $\lambda \in [0,1]$, is a family of nested sets, i.e. $\mathcal{C}(\lambda_1)\subset \mathcal{C}(\lambda_2)$ if $\lambda_1>\lambda_2$. By the continuity and monotonicity of $\underline{a}_j(\lambda)$, $\overline{a}_j(\lambda)$ and $\overline{\delta}(\lambda)$ we get
\begin{equation}
\label{en}
	{\rm N}(\mathcal{C}(\lambda))=1-\lambda,\;\; \lambda\in [0,1].
\end{equation}
Notice also that $\mathcal{C}(\lambda)$ is a closed convex set for each $\lambda\in [0,1]$ and has nonempty interior for all $\lambda\in [0,1)$.

\begin{prop}
	The following equality 
	$$\mathcal{P}(\pi)=\{{\rm P}\in \mathcal{PM}(\mathbb{R}^n) : {\rm P}(\mathcal{C}(\lambda))\geq 1-\lambda, \lambda\in [0,1]\}$$
	holds, where $\mathcal{PM}(\mathbb{R}^n)$ is the set of all probability measures on~$\mathbb{R}^n$.
\end{prop}
\begin{proof}
	Let $\mathcal{P}'(\pi)=\{{\rm P}\in \mathcal{PM}(\mathbb{R}^n) : {\rm P}(\mathcal{C}(\lambda))\geq 1-\lambda, \lambda\in [0,1]\}$. We need to show that $\mathcal{P}'(\pi)=\mathcal{P}(\pi)$. Observe first that $\mathcal{P}(\pi)\subseteq \mathcal{P}'(\pi)$. To see this choose any probability distribution ${\rm P}\in \mathcal{P}(\pi)$. Since $\mathcal{C}(\lambda)\subseteq \mathbb{R}^n$ is an event for each $\lambda\in [0,1]$, we get ${\rm P}(\mathcal{C}(\lambda))\geq {\rm N}(\mathcal{C}(\lambda))=1-\lambda$ for each $\lambda\in [0,1]$ (see equation~(\ref{en})). In consequence ${\rm P}\in \mathcal{P}'(\pi)$. We now show that $\mathcal{P}'(\pi)\subseteq \mathcal{P}(\pi)$.
Let ${\rm P}\in \mathcal{P}'(\pi)$ and consider event $A\subseteq \mathbb{R}^n$. If $\mathcal{C}(1)=\{\hat{\pmb{a}}\}\not\subseteq A$, then ${\rm N}(A)=0$ and ${\rm P}(A)\geq {\rm N}(A)=0$. Suppose $\mathcal{C}(1)\subseteq A$ and let
	$$\lambda^*=
	\inf \{\lambda\in [0,1] : \mathcal{C}(\lambda)\subseteq A\}.
	$$
	Then ${\rm P}(A)\geq {\rm P}(\mathcal{C}(\lambda^*))\geq 1-\lambda^*={\rm N}(\mathcal{C}(\lambda^*)).$
	Choose $\lambda':=\lambda^*-\epsilon$ for arbitrarily small $\epsilon>0$. 
	As $\mathcal{C}(\lambda')\not \subseteq A$,  there is scenario $\pmb{a}$ such that $\pmb{a}\notin A$ and $\pi(\pmb{a})\geq \lambda'$. But  $\pmb{a}\in A$ for each scenario $\pmb{a}$ such that $\pi(\pmb{a})\geq \lambda^*$. Consequently $\sup_{\pmb{a}\notin A} \pi(\pmb{a})\in [\lambda',\lambda^*]$. Because $\epsilon\rightarrow 0$, ${\rm N}(A)=1-\lambda^*$.
	 Hence ${\rm P}(A)\geq {\rm N}(A)$ and ${\rm P}\in \mathcal{P}(\pi)$.
\end{proof}

In order to construct a tractable reformulation of~(\ref{drc}) for $\mathcal{P}(\pi)$ we need to use a discretization of $\mathcal{P}(\pi)$.
Choose integer $\ell\geq 0$ and set $\Lambda=\{0,1,\dots,\ell\}$. Define $\lambda_i=i/\ell$, $i\in \Lambda$. We consider the following ambiguity set:
\begin{equation}
\mathcal{P}^{\ell}(\pi)=\{{\rm P}\in \mathcal{PM}(\mathbb{R}^n) : {\rm P}(\mathcal{C}(\lambda_i))\geq 1-\lambda_i: i\in \Lambda\}.
\label{asi}
\end{equation}
Set $\mathcal{P}^{\ell}(\pi)$ is a discrete approximation of $\mathcal{P}(\pi)$. It is easy to verify that $\mathcal{P}(\pi)\subseteq \mathcal{P}^{\ell}(\pi)$ for any $\ell\geq 0$. By fixing sufficiently large constant $\ell$, we obtain arbitrarily close approximation of $\mathcal{P}(\pi)$.

 \begin{figure*}[ht]
\centering
\includegraphics[width=10cm]{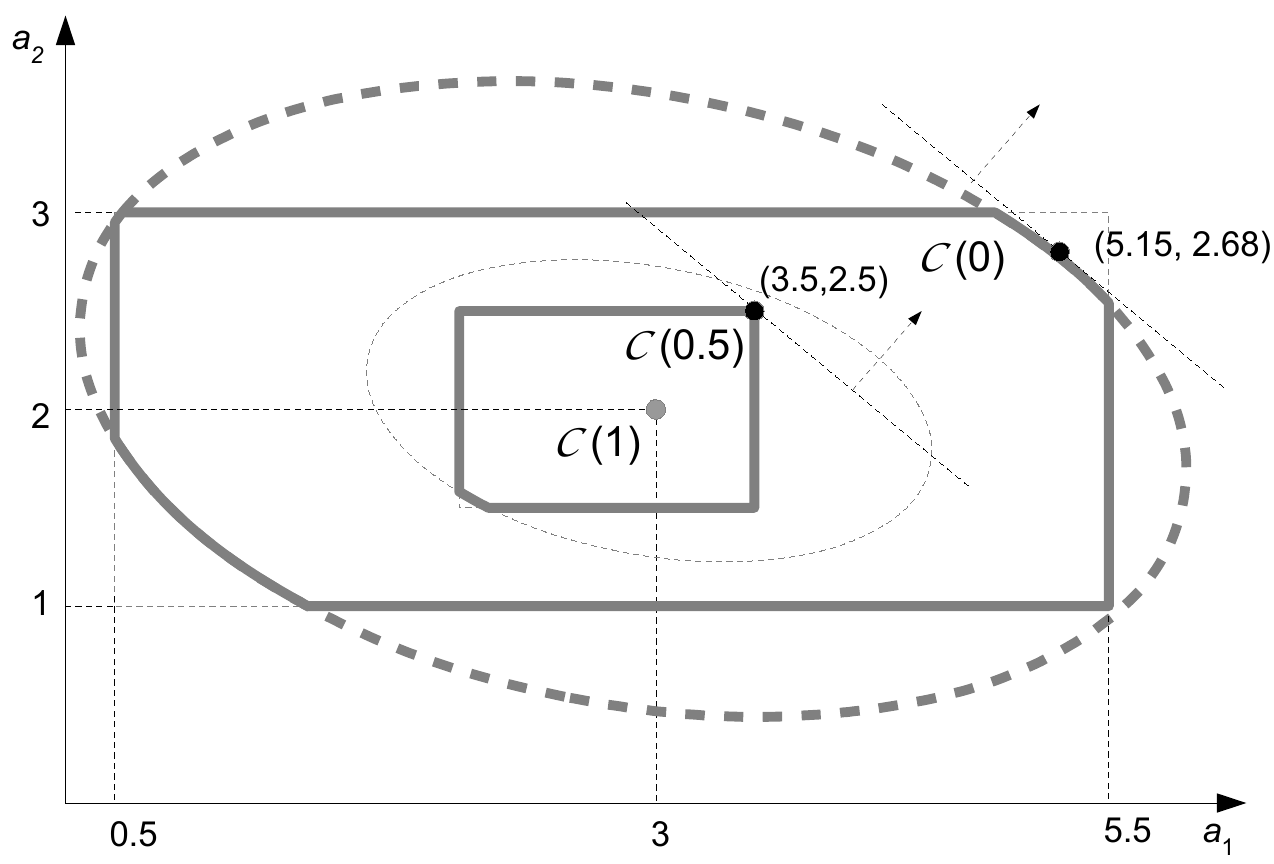}
\caption{The ambiguity set $\mathcal{P}^{2}(\pi)$, ${\rm P}(\mathcal{C}(0))\geq 1$, ${\rm P}(\mathcal{C}(0.5))\geq 0.5$, ${\rm P}(\mathcal{C}(1))\geq 0$} \label{fig2}
\end{figure*}

\begin{thm}
\label{thmmodint}
	The constraint
	\begin{equation}
\label{drc1}
\max_{{\rm P}\in \mathcal{P}^{\ell}(\pi)} \mathrm{\mathbf{E}}_{\rm P}[\tilde{\pmb{a}}^T\pmb{x}]\leq b
\end{equation}
is equivalent to the following system of 
second-order cone constraints:
\begin{equation}
\label{modfin}
		\begin{array}{llll}
			\displaystyle  w+\sum_{i\in \Lambda} (\lambda_i-1)v_i\leq b\\
				 \displaystyle  \gamma_i \overline{\delta}(\lambda_i) + \sum_{j\in [n]} \alpha_{ij} (\overline{a}_j(\lambda_i)-\hat{a}_j) + \sum_{j\in [n]} \beta_{ij}(\hat{a}_j-\underline{a}_j(\lambda_i))+\\
				 \begin{array}{lll}
				+ \hat{\pmb{a}}^T\pmb{x} \leq \displaystyle w-\sum_{j\leq i}  v_j &  i\in \Lambda \\
				  \alpha_{ij} - \beta_{ij} + \pmb{B}^{T}_j\pmb{u}_i=x_j & i\in \Lambda, j\in [n]\\
	 ||\pmb{u}_i||_{2}\leq \gamma_i & i\in \Lambda \\
	 \alpha_{ij}, \beta_{ij}\geq 0 &i\in \Lambda,  j\in[n]\\
	 v_i,\gamma_i \geq 0 & i\in \Lambda\\
	 w\in \Rset, \pmb{u}_i\in \mathbb{R}^n &i\in \Lambda
		 \end{array}
	\end{array}
\end{equation}
where $\Lambda=\{0,1,\dots,\ell\}$ and $\pmb{B}_j$ is the $j$th column of $\pmb{B}$.
\end{thm}
\begin{proof}
	The proof is 
adapted from~\cite{WKS14}.
The left hand side of~(\ref{drc1}) can be expressed as  the \emph{problem of moments}
(see,~e.g.,~\cite{SH01,WKS14}):
\begin{equation}
\label{mod1}
	\begin{array}{lll}
		\max &  \displaystyle \int_{\mathcal{C}(\lambda_0)} \pmb{a}^T\pmb{x} \; {\rm d}\mu(\pmb{a}) \\
			& \displaystyle \int_{\mathcal{C}(\lambda_0)} \pmb{1}_{[\pmb{a}\in \mathcal{C}(\lambda_i)]} \; {\rm d}\mu(\pmb{a})\geq 1-\lambda_i & i\in \Lambda\\
			& \displaystyle  \int_{\mathcal{C}(\lambda_0)} {\rm d}\mu(\pmb{a})=1\\
			& \mu \in \mathcal{M}_{+}(\mathbb{R}^n),
	\end{array}
\end{equation}
where $\mathcal{M}_{+}(\mathbb{R}^n)$ is the set of all nonnegative measures on~$\mathbb{R}^n$. Notice that the second (equality) constraint implies $\mu$ is a probability measure supported on $\mathcal{C}(\lambda_0)$.
The dual of  the problem of moments takes the following form (see, e.g.,~\cite{KL19}):
\begin{equation}
\label{mod2}
	\begin{array}{llll}
			\min  & \displaystyle w+\sum_{i\in \Lambda}(\lambda_i-1) v_i\\
				& \displaystyle w-\sum_{i\in \Lambda} \pmb{1}_{[\pmb{a}\in \mathcal{C}(\lambda_i)]} v_i\geq \pmb{a}^T\pmb{x} & \forall \pmb{a} \in \mathcal{C}(\lambda_0) \\
				& v_i\geq 0 & i\in \Lambda\\
				& w\in \Rset &	
\end{array}
\end{equation}
Strong duality
 implies that the optimal objective values of~(\ref{mod1}) and (\ref{mod2}) are the same
(see, e.g.,~\cite[Theorem~1, Corollary~1]{KL19}). Define 
$$\overline{\mathcal{C}}(\lambda_i)=\mathcal{C}(\lambda_i)\setminus \mathcal{C}(\lambda_{i+1}),\; i\in[\ell-1]$$
and $\overline{\mathcal{C}}(\lambda_{\ell})=\mathcal{C}(\lambda_{\ell})$.
Hence $\overline{C}(\lambda_i)$, $i\in \Lambda$, form a partition of the support $\mathcal{C}(\lambda_0)$ into $\ell+1$ disjoint sets.
Model (\ref{mod2}) can be then rewritten as follows:
\begin{equation}
\label{mod3}
	\begin{array}{llll}
			\min  & \displaystyle w+\sum_{i\in \Lambda}(\lambda_i-1) v_i\\
				& \displaystyle w-\sum_{i\in \Lambda} \pmb{1}_{[\pmb{a}\in \mathcal{C}(\lambda_i)]} v_i\geq \pmb{a}^T\pmb{x} & \forall \pmb{a} \in \overline{\mathcal{C}}(\lambda_i), i\in \Lambda \\
				& v_i\geq 0 & i\in \Lambda\\
				&w\in \Rset &
	\end{array}
\end{equation}
which is equivalent to
\begin{equation}
\label{mod4}
	\begin{array}{llll}
			\min  & \displaystyle w+\sum_{i\in \Lambda}(\lambda_i-1) v_i\\
				& \displaystyle w-\sum_{j\leq i}  v_j\geq \pmb{a}^T\pmb{x} & \forall \pmb{a} \in \overline{\mathcal{C}}(\lambda_i), i\in \Lambda \\
				& v_i\geq 0 & i\in \Lambda\\
				& w\in \Rset &
	\end{array}
\end{equation}
The $i$th constraint in~(\ref{mod4}) can be reformulated as
\begin{equation}
\label{f00}
\max_{\pmb{a} \in \overline{\mathcal{C}}(\lambda_i)} \pmb{a}^T\pmb{x}\leq w-\sum_{j\leq i} v_j.
\end{equation}
By the linearity of $\pmb{a}^T\pmb{x}$ in $\pmb{a}$, the left hand side is maximized at the boundary of $\overline{\mathcal{C}}(\lambda_i)$, which coincides with the boundary of $\mathcal{C}(\lambda_i)$. Hence~(\ref{f00}) is equivalent to
$$\max_{\pmb{a} \in \mathcal{C}(\lambda_i)} \pmb{a}^T\pmb{x}\leq w-\sum_{j\leq i} v_j.$$
The left hand side of this inequality, by  definition of $\mathcal{C}(\lambda_i)$, yields the following  problem:
$$
\begin{array}{lll}
	\max &  \pmb{a}^T\pmb{x}\\
		& a_j\leq \overline{a}_j(\lambda_i) & j\in [n]\\\
		& -a_j\leq -\underline{a}_j(\lambda_i) & j\in[n]\\
		&\displaystyle  ||\pmb{B}(\pmb{a}-\hat{\pmb{a}})||_{2}\leq \overline{\delta}(\lambda_i)&\\
		& \pmb{a}\in \Rset^n &
\end{array}
$$
Let us substitute $\pmb{y}=\pmb{a}-\hat{\pmb{a}}$, which yields:
$$
\begin{array}{lll}
	\max &  \pmb{y}^T\pmb{x}+\hat{\pmb{a}}^T\pmb{x}\\
		& y_j\leq -\hat{a}_j+\overline{a}_j(\lambda_i) & j\in [n]\\\
		& -y_j\leq -\underline{a}_j(\lambda_i)+\hat{a}_j & j\in [n]\\
		&\displaystyle  ||\pmb{B}\pmb{y}||_{2}\leq \overline{\delta}(\lambda_i)&\\
		&\pmb{y}\in \Rset^n &
\end{array}
$$
Introducing new variables $\pmb{z}$ leads to the following model:
\begin{equation}
\label{mod04}
\begin{array}{lll}
	\max &  \pmb{y}^T\pmb{x}+\hat{\pmb{a}}^T\pmb{x}\\
		& y_j\leq -\hat{a}_j+\overline{a}_j(\lambda_i) & j\in [n]\\\
		& -y_j\leq -\underline{a}_j(\lambda_i)+\hat{a}_j & j\in [n]\\
		& \pmb{B}\pmb{y}-\pmb{z}=\pmb{0}&\\
		&\displaystyle  ||\pmb{z}||_{2}\leq \overline{\delta}(\lambda_i)&\\
		&\pmb{y}, \pmb{z}\in \Rset^n &
\end{array}
\end{equation}
The dual to~(\ref{mod04}) is (see the Appendix):
\begin{equation}
\label{mod05}
\begin{array}{lll}
	\min & \displaystyle \gamma \overline{\delta}(\lambda_i) + \sum_{j=1}^n \alpha_j (\overline{a}_j(\lambda_i)-\hat{a}_j) + \sum_{j=1}^n \beta_j(\hat{a}_j-\underline{a}_j(\lambda_i))+ \\
	& +\hat{\pmb{a}}^T\pmb{x}\\
	& \begin{array}{ll}
		\alpha_j - \beta_j + \pmb{B}^{T}_j\pmb{u}=x_j & j\in [n]\\
		 ||\pmb{u}||_{2}\leq \gamma \\
		 \alpha_j, \beta_j,\gamma\geq 0 & j\in [n]\\
		 \pmb{u}\in \Rset^n &
		 \end{array}
\end{array}
\end{equation}
where $\pmb{B}_j$ is the $j$th column of $\pmb{B}$. The strong duality (see Appendix) implies that~(\ref{mod04}) and~(\ref{mod05}) have the same optimal objective function values.
Model~(\ref{mod05}) together with~(\ref{mod4}) yield~(\ref{modfin}).
\end{proof}

 Theorem~\ref{thmmodint} leads to the following corollary:
\begin{cor}
If $\Xset$ is a polyhedron described by a system of linear constraints
(i.e. (\ref{pf}) is an uncertain linear programming problem),
then the  deterministic counterpart of~(\ref{pf}) with the constraints of
the type (\ref{drc1}) is a second-order cone program.
\label{cordci}
\end{cor}

Corollary~\ref{cordci} shows that
the resulting deterministic counterpart of uncertain linear programming problem~(\ref{pf}),
under the interval uncertainty model,
 is polynomially solvable (see, e.g.,~\cite{BV08})
and thus can solve efficiently by
some standard  off-the-shelf solvers like \texttt{IBM ILOG CPLEX}~\cite{CPLEX}.

Consider a sample  of optimization problem~(\ref{pf}) in which the value of the objective function 
$\tilde{a}_1x_1+\tilde{a}_2x_2$ is uncertain and two constraints:  $x_1\geq 2.74$ and $x_2 \geq 3.3$
are precisely known.
The uncertainty of the  coefficients $\tilde{a}_1$ and $\tilde{a}_2$
is modeled by fuzzy intervals $\tilde{a}_1=\braket{3,2.5,2.5}_{1-0.32}$ and $\tilde{a}_2=\braket{2,1,1}_{1-1}$, respectively,
regarded as possibility distributions for their values. A possibility distribution for uncertain deviation $\tilde{\delta}$ is 
prescribed by fuzzy interval $\tilde{\delta}=\braket{0,0,6}_{1}$, $\Gamma=6$, and
$\pmb{B}=\left[\begin{array}{ll} 2 & 2.5 \\ 1 & -3\end{array}\right]$. By~(\ref{jpdi}),
we obtain 
a joint possibility distribution~$\pi$ for 
 $\tilde{\pmb{a}}=(\tilde{a}_1,\tilde{a}_2)$.
Fix $\ell=2$, which leads to the ambiguity set $\mathcal{P}^2(\pi)$ for $\tilde{\pmb{a}}=(\tilde{a}_1,\tilde{a}_2)$ 
shown in Figure~\ref{fig2},  built according to~(\ref{asi}). In particular, the set $\mathcal{C}(0)$, being the intersection of $[0.5,5.5]\times[1,3]$ and the ellipse $\{\pmb{a}: ||\pmb{B}(\pmb{a}-\hat{\pmb{a}})||_2\leq 6\}$ contains all possible scenarios.
Applying  now the distributionally robust approach in the possibilistic setting gives:
\begin{equation}
\label{drex}
	\begin{array}{lll}
		\min  &\displaystyle \max_{{\rm P} \in \mathcal{P}^{\ell}(\pi)}\mathrm{\mathbf{E}}_{\rm P}[\tilde{a}_1x_1+\tilde{a}_2x_2] \\
			& x_1\geq 2.74\\
			& x_2 \geq 3.3
	\end{array}
\end{equation}
Remark~\ref{rcbprec} shows how to convert 
the above problem with the uncertain objective  function to the one   with the precise objective
function and the uncertain constraint.
Theorem~\ref{thmmodint}  (see Corollary~\ref{cordci}) now leads to a second-order cone program
that is equivalent to (\ref{drex}).
An optimal solution of the program is $x_1=2.74$ and $x_2=3.3$. Let us focus on evaluating the objective function
$$\max_{{\rm P} \in \mathcal{P}^2(\pi)}\mathrm{\mathbf{E}}_{\rm P}[2.74\tilde{a}_1+3.3\tilde{a}_2].$$
The worst probability distribution ${\rm P}$ in $\mathcal{P}^2(\pi)$ assigns probability 0.5 to scenario $(5.15, 2.68)\in \mathcal{C}(0)$ and probability 0.5 to scenario $(3.5, 2.5)$ in $\mathcal{C}(0.5)$. These two points can be obtained by maximizing the linear function $2.74a_1+3.3a_2$ over the convex sets $\mathcal{C}(0)$ and $\mathcal{C}(0.5)$, respectively (see Figure~\ref{fig2}).
 We thus get $\max_{{\rm P} \in \mathcal{P}^2(\pi)}\mathrm{\mathbf{E}}_{\rm P}[2.74\tilde{a}_1+3.3\tilde{a}_2]=0.5\cdot(2.74\cdot 5.15+3.3\cdot 2.68)+0.5\cdot(2.74\cdot3.5+3.3\cdot 2.5)=20.39$.

%
%

\section{Risk aversion modeling}
\label{secrisk}

In this section we will show how the ambiguity set $\mathcal{P}^{\ell}(\pi)$ can be relaxed to take individual decision maker risk aversion into account.  Observe that the lower bounds on the probabilities for the sets $\mathcal{C}(\lambda_i)$ in $\mathcal{P}_{\pi}^{\ell}$ change uniformly. In consequence,  a worst probability distribution uniformly assigns probabilities to points on the boundaries of $\mathcal{C}(\lambda_0),\dots,\mathcal{C}(\lambda_{\ell-1})$. An example is shown in Figure~\ref{fig2}, where (for $\ell=2$) probability 0.5 is assigned to points on the boundaries of $\mathcal{C}(0)$ and $\mathcal{C}(0.5)$.
It has been observed that risk averse decision makers can perceive worst coefficient realizations as more probable (see, e.g.,~\cite{QU82, DW01}). We now propose an method of taking this risk-aversion into account.
 Let 
$$g_{\rho}(z)=\frac{1}{1-\rho}(1-\rho^z)$$ 
for $\rho\in (0,1)$ and $z\in [0,1]$. The function $g_{\rho}(z)$ is continuous, concave,  increasing in $[0,1]$ and such that $g_{\rho}(0)=0$ and $g_{\rho}(1)=1$. If $\rho\rightarrow 0$, then $g_{\rho}(z)\rightarrow 1$ for each $z\in (0,1]$. On the other hand, if $\rho\rightarrow 1$, then $g_{\rho}(z)\rightarrow z$ for each $z\in [0,1]$, so $g_{\rho}(z)$ tends to 
a linear function (see Fig.~\ref{figg}).

\begin{figure}[ht]
\centering
\includegraphics[width=6cm]{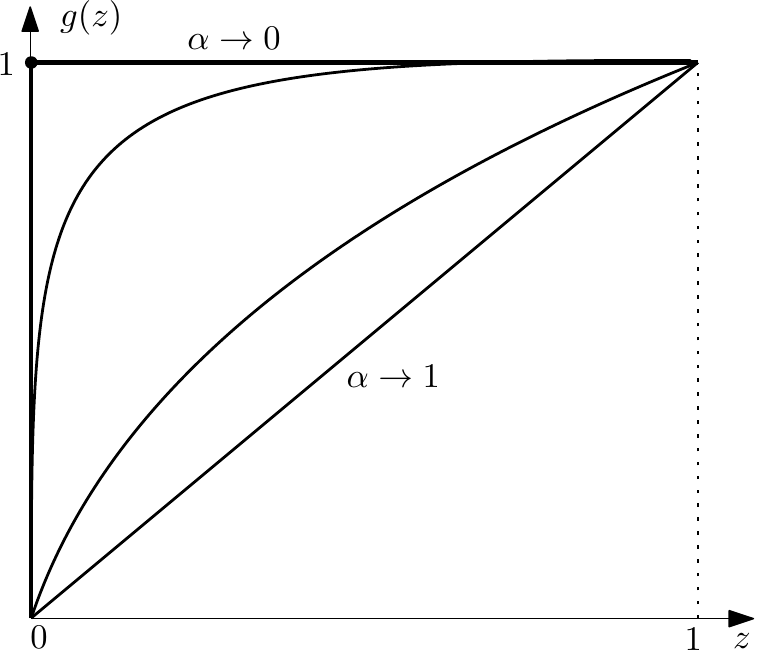}
\caption{Function $g_{\rho}(z)=\frac{1}{1-\rho}(1-\rho^z)$.} \label{figg}
\end{figure}

Fix $\rho\in (0,1)$ and  consider the following ambiguity set:
$$\mathcal{P}^{\ell,\rho}(\pi)=\left\{{\rm P}\in \mathcal{PM}(\mathbb{R}^n) : {\rm P}(\mathcal{C}(\lambda_i))\geq 1-g_{\rho}(\lambda_i),\; i\in \Lambda\right\}.$$
Since $1-\lambda_i\geq 1-g_{\rho}(\lambda_i)$, we still get a valid bound ${\rm P}(\mathcal{C}(\lambda_i))\geq {\rm N}(\mathcal{C}(\lambda_i))=1-\lambda_i\geq 1-g_{\alpha}(\lambda_i)$. Notice that the lower bound can be now smaller (i.e. larger probabilities can be assigned to worse scenarios), which reflects the probability distortion. 

As $\mathcal{P}_{\pi}^{\ell}\subseteq \mathcal{P}^{\ell, \rho}_{\pi}$ for each $\rho\in (0,1)$ and $\mathcal{P}^{\ell,\rho}_{\pi}\subseteq \mathcal{P}^{\ell,\rho'}_{\pi}$ if $\rho \geq \rho'$, decreasing $\rho$ leads to more conservative constraint~(\ref{drc}). If $\rho\rightarrow 1$, then $\mathcal{P}^{\ell,\rho}_{\pi}=\mathcal{P}_{\pi}^{\ell}$. On the other hand, if $\rho\rightarrow 0$, then $\mathcal{P}^{\ell, \rho}_{\pi}$ is the set of all probability measures with support $\mathcal{C}(0)$, and the constraint~(\ref{drc}) becomes 
$$\max_{\pmb{a}\in \mathcal{C}(0)} \pmb{a}^T\pmb{x}\leq b,$$
which is equivalent to the strict robust constraint~(\ref{rc}).
Theorem~\ref{thmmodint} remains valid for the ambiguity set $\mathcal{P}^{\ell, \rho}_{\pi}$. It is enough to replace the constant $\lambda_i$ with the constant $g_{\rho}(\lambda_i)$ in the first inequality in~(\ref{modfin}).

\section{Application to portfolio selection}
In this section we will apply the model constructed in Section~\ref{secint}
(see also Section~\ref{secrisk})
 to a portfolio selection problem.  We are given a set of $n$ assets whose returns form a random vector $\tilde{\pmb{a}}=(\tilde{a}_1,\dots,\tilde{a}_n)$. A portfolio is a vector $\pmb{x}\in [0,1]^n$, where $x_i$ is the share of the $i$th asset. Then $\tilde{\pmb{a}}^T\pmb{x}$ is the uncertain return of portfolio $\pmb{x}$ (accordingly, the quantity $-\tilde{\pmb{a}}^T\pmb{x}$ is the uncertain loss for $\pmb{x}$).
The probability distribution for $\tilde{\pmb{a}}$ is unknown. However, we have a sample of past realizations $\pmb{a}_1,\dots,\pmb{a}_K$ of $\tilde{\pmb{a}}$. Using this sample we can estimate the mean return $\hat{\pmb{a}}$  and the covariance matrix $\pmb{\Sigma}$ for $\tilde{\pmb{a}}$. The sample mean and covariance matrix for~7 assets, computed for a sample of 30 subsequent observations in Polish stock market, are 

$$
	\hat{\pmb{a}}=[0.057, -0.378, 0.324, -0.799, -0.873, -0.271, -0.323]
$$

$$
{
\footnotesize
\arraycolsep=1.6pt
\pmb{\Sigma}=\left[ \begin{array}{rrrrrrrr}
					7.469 & 0.149 & 0.099 & 0.076 & 2.225 & 0.044 & 1.649 \\
					 &  0.967 & 0.865 & -0.578 & -1.558 & 0.053 & -0.143\\
					 &	        & 3.714 & -0.454 & -1.265 & 1.188 & 0.320 \\
					 &              &	     & 2.188 & -0.529 & -0.152 & 0.525 \\
					 &		& &			  & 18.168 & -1.561 & 4.558 \\
					 &		&		&        &		& 12.745 & 1.391\\
					 &		&		&	&		&		& 5.371
			  \end{array}\right] }
$$

We fix $\pmb{B}=\pmb{\Sigma}^{\frac{1}{2}}$,  $\ell=100$, $\underline{a}_j=\overline{a}_j=6\cdot\sigma_j$, where $\sigma_j=\sqrt{\Sigma_{jj}}$ for  each $j\in [n]$. Using Chebyshev's inequality one can show that $\tilde{a}_j\in [\hat{a}_j-\underline{a}_j, \hat{a}_j+\overline{a}_j]$ with high probability.  We now use the fuzzy intervals $\braket{\hat{a}_j, \underline{a}_j, \overline{a}_j}_{1-1}$ to model the possibility distributions for the asset returns. We use linear membership functions. One can, however, further refine the information for the return values using different shapes, i.e. changing the parameters $z_1$ and $z_2$.  We use the fuzzy interval $\braket{0,\Gamma}_1$ for $\Gamma\in [0,50]$ to model the possibility distribution for the deviation $\tilde{\delta}$.
We consider the following problem:

\begin{equation}
\label{psf}
	\begin{array}{lllll}
		\displaystyle \min  & \displaystyle \max_{{\rm P}\in \mathcal{P}^{\ell}(\pi)} \mathrm{\mathbf{E}}_{\rm P}[-\tilde{\pmb{a}}^T\pmb{x}] \\
					    &\pmb{1}^T\pmb{x} =1 \\
					    &\pmb{x}\geq \pmb{0}
	\end{array}
\end{equation}

 \begin{figure}[ht]
\centering
\includegraphics[width=8cm]{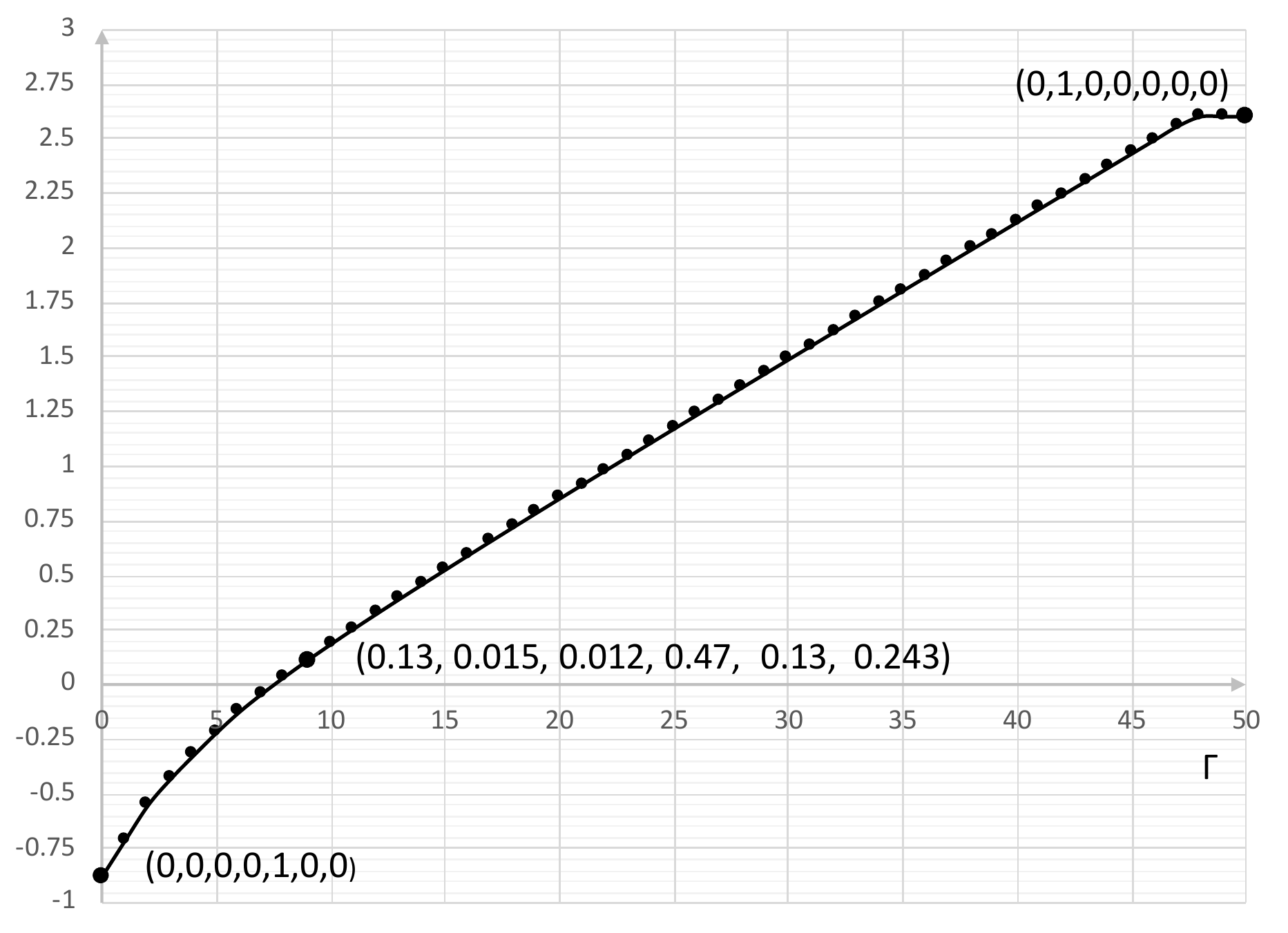}
\caption{Optimal objective values for various $\Gamma$.} \label{fig4}
\end{figure}

Figure~\ref{fig4} shows the optimal objective value of~(\ref{psf}) for $\Gamma=0,1,\dots,50$. In the first boundary case, $\Gamma=0$, there is no uncertainty and in the optimal portfolio all is allocated to the asset with the largest expected return (smallest expected loss). In the second boundary case, when $\Gamma\geq 48$, there is no limit on deviation from $\hat{\pmb{a}}$ and all is allocated to the asset with the smallest $\hat{a}_j+\overline{a}_j$. For the intermediate values of $\Gamma$, we get a family of various portfolios in which a diversification is profitable. One such portfolio is shown in Figure~\ref{fig4}.

Figure~\ref{fig5} shows the effect of taking the individual risk aversion into account. In this figure the optimal objective value of~(\ref{psf}) for $\rho\in (0,1)$ is shown (we fix $\Gamma=20$). For smaller values of $\rho$ the objective value of~(\ref{psf}) is larger. One can see that the optimal portfolio is adjusted (see Figure~\ref{fig5}) to take larger sets of admissible probability distributions into account.

 \begin{figure}[ht]
\centering
\includegraphics[width=8cm]{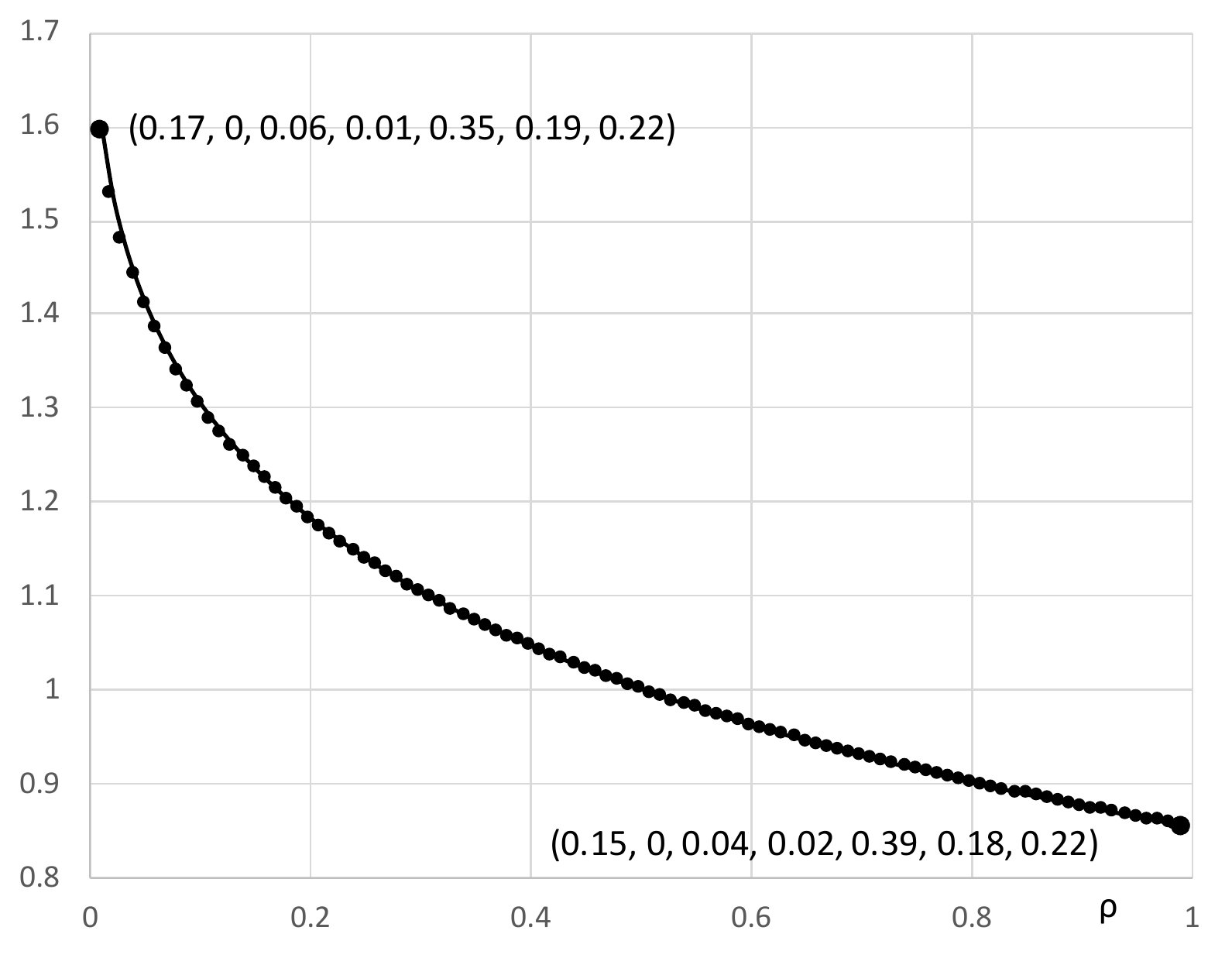}
\caption{Optimal objective values for various $\rho$ and $\Gamma=20$.} \label{fig5}
\end{figure}

\section{Conclusions}

In this paper we have proposed a method of unifying the fuzzy (possibilistic) optimization with the distributionally robust approach.  Using the connections between the possibility and probability theories one can build a set of probability distributions for uncertain constraint coefficient, based on a possibilistic information. The uncertain constraint, the  left-hand side, can be then replaced with the expected value with respect to the worst probability distribution that can occur. In practical applications the true probability distributions for the uncertain data are often unknown. However, in most cases some information about the possible scenarios (parameter realizations) is available. In this case the possibility theory is an attractive choice, because it is flexible and allows to take both available data and experts' opinions into account. In this paper we have proposed some methods of defining possibility distributions for scenarios, which is easy to apply in practice. Furthermore, the resulting deterministic counterpart of the problem can be solved efficiently for a large class of problems (for example for linear programming).

The key element of our model is the definition of a joint possibility distribution in scenario set. This step is quite flexible, as there are only minor restrictions on possibility distribution function which must be imposed. The support of this possibility distribution should contain all possible scenarios and a budget should control the magnitude of deviations of scenarios from some nominal (expected) one. We believe that such possibility distribution can be a good tool to handle the uncertainty in optimization problems.


\section*{Appendix}

\begin{proof}[The proof of~(\ref{mod05})]
Let us rewrite~(\ref{mod04}) as follows:

\begin{equation}
\label{mod04a}
\begin{array}{lll}
	\min &  -\pmb{y}^T\pmb{x}-\hat{\pmb{a}}^T\pmb{x}\\
		& y_j\leq -\hat{a}_j+\overline{a}_j(\lambda_i) & j\in [n]\\\
		& -y_j\leq -\underline{a}_j(\lambda_i)+\hat{a}_j & j\in [n]\\
		& \pmb{B}\pmb{y}-\pmb{z}=\pmb{0}&\\
		&\displaystyle  ||\pmb{z}||_{2}\leq \overline{\delta}(\lambda_i)&\\
		&\pmb{y}, \pmb{z}\in \Rset^n&
\end{array}
\end{equation}
and write the Lagrangian function
$$
\begin{array}{lllll}
	L(\pmb{y},\pmb{z},\pmb{\alpha},\pmb{\beta},\pmb{u},\gamma)= &  -\pmb{y}^T\pmb{x}-\hat{\pmb{a}}^T\pmb{x} + \\
	& \displaystyle \sum_{j\in [n]} \alpha_j(y_j+\hat{a}_j- \overline{a}_j(\lambda_i))+ \\
	& \displaystyle \sum_{j\in [n]} \beta_j(-y_j+\underline{a}_j(\lambda_i)-\hat{a}_j)+\\
	&(\pmb{B}\pmb{y}-\pmb{z})^T\pmb{u}+\gamma(||\pmb{z}||_{2}-\overline{\delta}(\lambda_i)),
\end{array}
$$
where $\alpha_j,\beta_j, \gamma\geq 0$ and $\pmb{u}\in\Rset^n$ are the Lagrangian multipliers.
After rearranging the terms we get
$$
	\begin{array}{lll}
	L(\pmb{y},\pmb{z},\pmb{\alpha},\pmb{\beta},\pmb{u},\gamma)= & -\hat{\pmb{a}}^T\pmb{x}+\\
	& \displaystyle \sum_{j\in [n]} (-x_j+\alpha_j-\beta_j+\pmb{B}^T_j\pmb{u})y_j +  \\
	& (\gamma ||\pmb{z}||_2-\pmb{z}^T\pmb{u})- \\
	& \displaystyle \sum_{j\in [n]} \alpha_j(\hat{a}_j-\overline{a}_j(\lambda_i))+ \\
	& \displaystyle \sum_{j\in [n]} \beta_j(\underline{a}_j(\lambda_i)-\hat{a}_j)- \gamma\overline{\delta}(\lambda_i).
\end{array}
$$
Problem~(\ref{mod04a}) is convex and the set of feasible solutions is bounded and has nonempty interior. Hence it has an optimal solution with the finite objective function value $v^*$. The problem also satisfies the strong duality
(see, e.g.,~\cite{BV08}), i.e.
\begin{equation}
\label{f001}
\max_{\pmb{\alpha}\geq \pmb{0},\pmb{\beta}\geq \pmb{0},\pmb{u},\gamma\geq 0} \min_{\pmb{y},\pmb{z}} L(\pmb{y},\pmb{z},\pmb{\alpha},\pmb{\beta},\pmb{u},\gamma)=v^*.
\end{equation}
Since $y_j$, $j\in [n]$, are unrestricted, the equality:
\begin{equation}
\label{f0}
-x_j+\alpha_j-\beta_j+\pmb{B}^T_j\pmb{u}=0, \; j\in [n]
\end{equation}
must hold. We now show that $\min_{\pmb{z}} (\gamma ||\pmb{z}||_2-\pmb{z}^T\pmb{u}) =0$ if $||\pmb{u}||_2\leq \gamma$ and $-\infty$, otherwise. Assume, that $||\pmb{u}||_2\leq \gamma$. By Cauchy-Shwartz inequality, $\pmb{z}^T\pmb{u}\leq ||\pmb{z}||_2 ||\pmb{u}||_2\leq \gamma ||\pmb{z}||_2$. Hence $\gamma ||\pmb{z}||_2-\pmb{z}^T\pmb{u}\geq 0$ and $\min_{\pmb{z}} (\gamma ||\pmb{z}||_2-\pmb{z}^T\pmb{u}) =0$. Assume that $||\pmb{u}||_2> \gamma$ and take $\pmb{z}=s\pmb{u}$ for some $s>0$. Then $\gamma ||\pmb{z}||_2-\pmb{z}^T\pmb{u}=\gamma s ||\pmb{u}||_2-s\pmb{u}^T\pmb{u}=s||\pmb{u}||_2(\gamma-||\pmb{u}||_2)$, which can be made arbitrarily small by increasing $s$. Hence, it must be 
\begin{equation}
\label{f1}
||\pmb{u}||_2\leq \gamma.
\end{equation}
Now, using~(\ref{f0}) and~(\ref{f1}), the problem~(\ref{f00}) can be rewritten as
$$
	\begin{array}{lllll}
		\max & \displaystyle -\hat{\pmb{a}}^T\pmb{x}+\sum_{j=1}^n\alpha_j(\hat{a}_j-\overline{a}_j(\lambda_i))+\sum_{j=1}^n \beta_j(\underline{a}_j(\lambda_i)-\hat{a}_j)- \\
		& \gamma\overline{\delta}(\lambda_i) \\
			& \begin{array}{lll} 
				 -x_j+\alpha_j-\beta_j+\pmb{B}^T_j\pmb{u}=0 &  j\in [n] \\
				||\pmb{u}||_2\leq \gamma& \\
			     \alpha_j, \beta_j, \gamma\geq 0 & j\in [n]\\
			     \pmb{u}\in \Rset^n &
			     \end{array}
	\end{array}
$$
which is equivalent to~(\ref{mod05}).
\end{proof}

\end{document}